\documentclass[11pt,reqno,a4paper]{amsart}

\usepackage{amssymb,amsfonts,amsmath,amsthm}
\usepackage{mathrsfs}
\usepackage{graphicx}
\usepackage{color}
\usepackage{verbatim}
\usepackage[hidelinks]{hyperref}
\usepackage{pdfsync}
\usepackage{mathtools}
\usepackage{todonotes}
\usepackage[all]{xy}
\usepackage{caption}
\usepackage{subcaption}
\usepackage{tikz}

\newtheorem{thm}{Theorem}[section]
\newtheorem{pro}[thm]{Proposition}
\newtheorem{lem}[thm]{Lemma}
\newtheorem{cor}[thm]{Corollary}

\theoremstyle{definition}

\newtheorem*{ackn}{Acknowledgements}

\theoremstyle{remark}
\newtheorem{rmk}[thm]{Remark}
\newtheorem{exa}[thm]{Example}
\newtheorem*{claim}{Claim}
\newtheorem{prb}[thm]{Problem}

\numberwithin{equation}{section}

\def\N{\mathbb{N}}

\def\R{\mathscr{R}}
\def\L{\mathscr{L}}

\def\es{\varnothing}

\def\Lra{\Leftrightarrow}
\def\mm{\mathfrak{m}}
\def\ol#1{\overline{#1}}

\def\pre#1#2{\langle{#1}\,|\,{#2}\rangle}

 \DeclareMathOperator\Inv{Inv} \DeclareMathOperator\Gp{Gp} 
  \DeclareMathOperator\red{red}

\DeclareMathOperator\Cay{Cay}

\begin{document}


\title[On Special Inverse Monoids with the Strong $F$-Inverse Property]{On Special Inverse Monoids with \\ the Strong $F$-Inverse Property} 


\author{IGOR DOLINKA}

\address{Department of Mathematics and Informatics, University of Novi Sad, Trg Dositeja Obra\-do\-vi\-\'ca 4,
21101 Novi Sad, Serbia}

\email{dockie@dmi.uns.ac.rs}

\author{GANNA KUDRYAVTSEVA}

\address{Faculty of Mathematics and Physics, University of Ljubljana, Jadranska ulica 19, SI-1000 Ljubljana, Slovenia / Institute of Mathematics, Physics and Mechanics, 
Jadranska ulica 19, SI-1000 Ljubljana, Slovenia}

\email{ganna.kudryavtseva@fmf.uni-lj.si}


\begin{abstract}
An inverse monoid $S$ is called $F$-inverse if each $\sigma$-class of $S$, where $\sigma$ is the minimum group congruence of $S$, has a maximum element with respect
to the natural order of $S$. Since the property of an inverse monoid being $F$-inverse immediately implies that it must be $E$-unitary, it follows that every 
$X$-generated $F$-inverse monoid with canonical maximum group image $G$ must be isomorphic to a quotient of the Margolis-Meakin expansion $M(G,X)$. If this is
realised in such a way that all the maximal elements of each $\sigma$-class of $M(G,X)$ get identified, thus producing the top element of the corresponding 
$\sigma$-class of $S$, we say that $S$ is strongly $F$-inverse. Consequently, there is a universal $X$-generated inverse monoid $M_{sF}(G,X)$ with maximum group 
image $G$ and the strongly $F$-inverse property. We provide a presentation for this inverse monoid and show it can be further simplified upon introducing additional
assumptions on the group $G$ (which will include all one-relator groups). We use this to provide a full description of all one-relator special inverse monoids with
a cyclically reduced relator word that are strongly $F$-inverse. We also discuss some further examples and non-examples.
\end{abstract}






\maketitle 


\section{Introduction} \label{sec:intro}

If $S$ is a semigroup or a monoid and $a\in S$ we say that $a'\in S$ is a \emph{(von Neumann) inverse} of $a$ if $aa'a=a$ and $a'=a'aa'$. If every element 
of a semigroup has an inverse such a semigroup is called \emph{regular}; regular semigroups in which the inverse of every element is unique are called
\emph{inverse semigroups}. In such a case, this unique inverse is usually denoted by $a^{-1}$ and the mapping $a\mapsto a^{-1}$ readily becomes an additional
unary operation of the underlying semigroup. Indeed, it is very convenient to consider inverse semigroups/monoids in this enriched signature because it is
well known that the class of inverse semigroups in this signature forms a variety, defined by the identical laws $(xy)^{-1}=y^{-1}x^{-1}$, $(x^{-1})^{-1}=x$, 
$xx^{-1}x=x$ and $xx^{-1}yy^{-1}=yy^{-1}xx^{-1}$, with the latter effectively telling us that idempotent elements commute and thus form a semilattice.
Inverse semigroups have nowadays a rich and developed theory \cite{Law,Pet}, forming a bridge between groups and semigroups \cite{Mea}, and making crucial 
appearances in many mathematical areas such as geometry, topology, category theory, and even outside mathematics, e.g.\ in the formal structural 
analysis in music theory \cite{BJM,Janin}. Also, inverse monoids are today an indispensable tool in studying the word problem for one-relator monoids \cite{CF}.
The question of whether every such monoid has an algorithmically soluble word problem amazingly still stands open even though the analogous problem for 
one-relator groups was famously solved in the positive by Magnus \cite{Ma1,Ma2} almost a century ago. It was discovered in \cite{IMM} that (special) inverse 
monoids and their word problem constitute one viable angle of attack to this question, although the major paper \cite{Gr-Inv} points towards the limitations 
of such an approach and necessary care in its applications.

Every inverse semigroup carries a natural partial order defined by 
$$
x\leq y \text{\ \ \ if and only if\ \ \ } x=xx^{-1}y
$$ 
(the latter being equivalent to an array of similar equivalent conditions). This order is compatible both with multiplication and the inversion operations. 
On the other hand, it is easy to show that the intersection of any family of congruences $\rho$ of the inverse semigroup/monoid $S$ with the property that 
$S/\rho$ is a group again has the same property, and hence there is a \emph{minimum group congruence} $\sigma$ on $S$, so that $S/\sigma$ is the \emph{maximum 
group homomorphic image} of $S$. Clearly, we may consider the restrictions of the order $\leq$ to each $\sigma$-class of $S$. If it turns out that each such 
$\sigma$-class in an inverse monoid $S$ has a maximum element then we say that $S$ is \emph{$F$-inverse}. An inverse monoid $S$ is called \emph{$E$-unitary} 
if the $\sigma$-class of the identity element consists precisely of the idempotents of $S$; now, $a\,\sigma\,1$ and $a\leq 1$ imply that $a$ is an idempotent, 
and so every $F$-inverse inverse monoid must necessarily be $E$-unitary. The $E$-unitary property is ubiquitous in inverse semigroup theory; loosely speaking, 
the $E$-unitary inverse monoids are the ones that are considered, in a way, ``well-behaved'', more closely related to groups than those lacking this property.

Generally, there is a recent surge of interest in $F$-inverse semigroups and monoids. Initially, it was fueled by the influential paper \cite{HR} of Henckell 
and Rhodes who conjectured that every finite inverse monoid $M$ has a finite $F$-inverse cover, that is, there is an idempotent separating homomorphism from 
a finite $F$-inverse monoid onto $M$. This conjecture was confirmed only very recently by Auinger, Bitterlich and Otto in \cite{ABO}. Also, there is a separate
strand of research stemming from the 2018 remark of M.\ Kinyon \cite{Kinyon} that upon introducing an additional unary operation $a\mapsto a^\mm$, which maps 
an element $a$ of an $F$-inverse monoid to the maximum element in the $\sigma$-class of $a$, the class of $F$-inverse monoids becomes a variety in this enriched 
signature. We refer to \cite{AKSz,KLF,Nora} as a sampler of the recent contributions in this vein. Finally, we mention the quite recent paper of Kambites and 
Szak\'acs \cite{KSz} which amply demonstrates that the property of an inverse monoid $S$ being $F$-inverse is, at its core, a geometric property of the Cayley 
graph of the greatest group image $G$ of $S$, or, put more precisely, a particular way in which the geometries of the Cayley graph of $G$ and 
the Sch\"utzenberger graphs of $S$ interact. 

In their seminal paper \cite{MM89}, Margolis and Meakin construct an inverse monoid $M(G,X)$ starting from an $X$-generated group $G$, consisting of finite 
connected pointed subgraphs of the Cayley graph of $G$ with respect to $X$. This is what later became known as the \emph{Margolis-Meakin expansion} of $G$, and 
it has the following remarkable property \cite[Theorem 2.2]{MM89}: it is the universal $X$-generated $E$-unitary inverse monoid with (canonical) maximum group 
image $G$. As it turns out, the maximal elements in the $\sigma$-class of $M(G,X)$ mapping onto $g\in G$ correspond precisely to simple paths $1\leadsto g$ in the 
Cayley graph of $G$. On the other hand, if $S$ is an $X$-generated $F$-inverse monoid such that $S/\sigma\cong G$ then there is a canonical morphism $\phi:M(G,X)\to S$,
and it respects the natural order and the $\sigma$-classes. So, the images of simple paths $1\leadsto g$ in $M(G,X)$ are mapped onto a partially ordered set 
of elements of the corresponding $\sigma$-class of $S$ such that (at least) one such image is, in the natural order of $S$, above all the other such images; 
this exactly captures the condition that $S$ is $F$-inverse. In this paper, we are interested in the situation when the $F$-inverse property of $S$ arises 
in a specific way, when \emph{all} maximal elements of a $\sigma$-class of $M(G,X)$ are mapped via $\phi$ to the unique top element 
of the corresponding $\sigma$-class of $S$. If such a situation occurs we say that the inverse monoid $S$ is \emph{strongly $F$-inverse}.

It this thus immediately clear that there exists a universal $X$-generated inverse monoid with maximum group image $G$ and the strongly $F$-inverse property, 
which we denote by $M_{sF}(G,X)$: it suffices to consider the quotient of $M(G,X)$ by the congruence identifying all simple paths from $1$ to a fixed element 
of $G$. After the next, preliminary section, in Section \ref{sec:univ} we aim to determine a manageable inverse monoid presentation for this object, one that 
is suitable for applications. We also provide a geometric model for $M_{sF}(G,X)$ based on a suitable closure operator on certain subgraphs of $\Cay(G,X)$. When the
group $G$ admits a presentation in which every relator word defines a simple cycle in its Cayley graph (by \cite{W2} this includes all one-relator groups),
we simplify this presentation further in Section \ref{sec:sims}. We then use the obtained information to provide a characterisation of all one-relator special inverse 
monoids $\Inv\pre{X}{w=1}$ with a cyclically reduced relator word $w$ that are strongly $F$-inverse. In the concluding section we discuss, in light of \cite{KSz}, 
special inverse monoids with a cyclically reduced relator word that are not $F$-inverse, and also examples of such monoids that are $F$-inverse but not strongly.
We finish with some open questions.


\section{Preliminaries} \label{sec:prelim}

In an attempt to keep this article reasonably self-contained, in this section we gather a ``toolbox'' of some concepts and prerequisites needed for our exposition. 
For other, more basic material, we are going to provide appropriate references.

\subsection{Basics of inverse semigroups} \label{subsec:basic}

A fundamental tool in studying the structure of semigroups are five equivalence relations called \emph{Green's relations}; here we define two of them. Namely, 
for a semigroup $S$ we write:
$$
a\,\R\,b \Lra aS^1=bS^1, \quad a\,\L\,b \Lra S^1a=S^1b,
$$
where $S^1$ denotes $S\cup\{1\}$ unless $S$ is already a monoid. In regular semigroups, each $\R$-class and each $\L$-class contains at least one idempotent.
The subclass of inverse semigroups is equivalently defined by the condition that each $\R$-class and each $\L$-class contains \emph{precisely one} idempotent,
namely, for all $a\in S$, where $S$ is an inverse semigroup, we have
$$ d(a)=aa^{-1}\;\R\; a \;\L\; r(a)=a^{-1}a.$$
Consequently, $a\,\R\,b$ if and only if $aa^{-1}=bb^{-1}$ and also $a\,\L\,b$ if and only if $a^{-1}a=b^{-1}b$. An element $a$ of an inverse monoid $S$ is called 
\emph{right invertible} (or a \emph{right unit}) if $d(a)=1$; similarly, $a\in S$ is \emph{left invertible} if $r(a)=1$. An element $a\in S$ is \emph{invertible} 
(or a \emph{unit}) if it is both right and left invertible. Invertible elements of an inverse monoid $S$ form its subgroup, the \emph{group of units} $U_S$.
 
Similarly as groups arose from the abstraction of permutations and semigroups from that of transformations of a set, inverse semigroups are abstract models
of partial injective maps on a set. Therefore, a prime exaple of an inverse semigroup (and, in fact, an inverse monoid) is the \emph{symmetric inverse monoid}  
$\mathcal{I}_A$ on $A$, consisting of of all partial injective mappings on $A$ with the usual operations of composition of partial mappings and of taking inverse 
of a partial injection. Unlike the morphisms in this paper (which are composed from right to left and written on the left of their arguments), it is customary 
to compose mappings in $\mathcal{I}_A$ from left to right. In particular, if $\alpha\in\mathcal{I}_A$ we have that $d(\alpha)=\alpha\alpha^{-1}$ 
is the identity mapping on the \emph{domain} of $\alpha$, and $r(\alpha)=\alpha^{-1}\alpha$ is the identity mapping on the \emph{range} of $\alpha$ (thus explaining 
the notation for these idempotents). In general, the idempotents of $\mathcal{I}_A$ are exhausted by the identity mappings on subsets of $A$, and consequently 
$\alpha\,\R\,\beta$ if and only if $\alpha,\beta$ have the same domain, while $\alpha\,\L\,\beta$ if and only if $\alpha,\beta$ have the same range (image).
By the \emph{Wagner-Preston theorem}, every inverse monoid/semigroup $S$ is isomorphic to an inverse submonoid/subsemigroup of $\mathcal{I}_S$.

For further background in semigroup theory we refer to \cite{CP,How,Law,Pet}.

\subsection{Inverse monoid presentations, invertible pieces} \label{subsec:inv}

If $X\neq\es$ is an alphabet, for the sake of representing and considering elements of groups, inverse monoids, and involution semigroups in general, it is useful 
to introduce the ``doubled'' alphabet $\ol{X}=X\cup X^{-1}$, where $X^{-1}$ is a copy of $X$ disjoint from $X$. Then $\ol{X}^*$ is the free involution monoid, 
meaning that every map $\iota:X\to S$ where $S$ is an involution monoid uniquely extends to a homomorphism of involution monoids $\ol{X}^* \to S$. If $\iota(X)$ 
generates $S$ we say that (the involution monoid\;/\; inverse monoid\;/\; group) $S$ is \emph{$X$-generated} (via $\iota$, which is in general not required to be 
injective). In such a case, every element of $S$ is represented by some word $w\in\ol{X}^*$; to distinguish between words as such and elements of $S$ represented 
by them we write $[w]_S$ for the element of $S$ represented by $w$, or just $[w]$ if the structure in question is clear from the context. For two $X$-generated 
structures $S_1,S_2$ via $\iota_1,\iota_2$, respectively, we say that a morphism $\psi:S_1\to S_2$ is \emph{canonical} if it respects the generators, that is 
if $\iota_2=\psi\iota_1$.

A word $w\in\ol{X}^*$ is \emph{reduced} if it does not contain a subword of the form $xx^{-1}$, $x^{-1}x$, $x\in X$. Otherwise, if a word $w$ is not reduced, it 
is easy to show that the process of successively removing subwords of the indicated form from $w$ is confluent and thus terminates by the unique \emph{reduced form} 
$\red(w)$ of $w$. As is well-known, the collection of all reduced words over $\ol{X}$ forms the \emph{free group} $FG(X)$ on $X$ under the operations 
$u\cdot v=\red(uv)$ and $u^{-1}$ obtained from $u$ by reversing the order of letters and then applying $^{-1}$ to each of them. A word $w\in\ol{X}^*$ is 
\emph{cyclically reduced} if it is reduced and the first and the last letters of $w$ are not mutually inverse.

Since inverse monoids form a variety, \emph{free inverse monoids} $FIM(X)$ exist for all alphabets $X$. This is the quotient of $\ol{X}^*$ by the 
so-called \emph{Wagner congruence}, generated by the pairs $(uu^{-1}u,u)$ and $(uu^{-1}vv^{-1},vv^{-1}uu^{-1})$ for all $u,v\in \ol{X}^*$. There is an elegant 
geometric model for $FIM(X)$, due to Munn \cite{Munn} and Scheiblich \cite{Sch}, which we will see a bit later, in Subsection \ref{subsec:mgx}.

Paralleling the notion of a \emph{group presentation} \cite{LSch} $\Gp\pre{X}{w_i=1\; (i\in I)}$ defining a quotient of the free group $FG(X)$ by the normal subgroup
generated by the words by $w_i$, $i\in I$, called \emph{relators} (usually assumed to be cyclically reduced), there is also the notion of an inverse monoid
defined by a presentation. Namely, we write
$$ M=\Inv\pre{X}{u_i=v_i\; (i\in I)} $$
if $M$ is isomorphic to the quotient of $FIM(X)$ by its congruence generated by $\{(u_i,v_i):\ i\in I\}$. Note that, unlike for groups, in the environment of 
inverse monoid presentations we are neither allowed to cancel adjacent mutually inverse letters or words nor is the \emph{defining relation} $u_i=v_i$ equivalent 
to $u_iv_i^{-1}=1$. In the case when $M$ admits a presentation by relations in which one of the words involved is the empty word $1$ we say that $M$ is \emph{special}. 
A group or an inverse monoid is \emph{finitely presented} if it admits a presentation by finitely many generators and relations. 

We refer to \cite{Stephen} for the basic notions and techniques related to the theory of inverse monoid presentations, and to \cite{LSch} for further background in
combinatorial group theory.

Now let $M=\Inv\pre{X}{w_i=1\; (i\in I)}$ be a special inverse monoid. Then, for every word $w_i$, $i\in I$, there is a factorisation
$$
w_i = w_{i,1}\cdots w_{i,k_i}
$$
with the properties that (i) each subword $w_{i,j}$ represents an invertible element of $M$, and (ii) no proper prefix or proper suffix of $w_{i,j}$ represents 
an invertible element of $M$. In this sense we speak about the \emph{finest} factorisation of $w_i$ into subwords representing invertible elements of $M$, 
which are called the \emph{minimal invertible pieces} or simply the \emph{pieces} of $M$. 
It follows from \cite[Proposition 4.2]{IMM} that elements represented by the  pieces 
of $M$ generate $U_M$, the group of units of $M$. The recent paper \cite{GR} goes further by describing presentations for such groups of units. We also refer 
to the same paper for an extensive discussion of many contrasts regarding the pieces and properties of groups of units with respect to ordinary monoid presentations 
and special monoids. In particular, while there is the \emph{Adyan overlap algorithm} \cite{DG21,GR,CF} for computing the pieces of finitely presented special monoids, 
it is not known if such an algorithm even exists for finitely presented special inverse monoids. 

\subsection{Graphs and Cayley graphs} \label{subsec:graphs}

In this paper all graphs considered are labelled digraphs $\Gamma=(V,E,\ell)$, where $V$ is the set of \emph{vertices}, $E$ is the set of \emph{edges} such that each 
edge $e$ has an initial vertex $\alpha(e)\in V$ and a terminal vertex $\omega(e)\in V$, while $\ell:E\to Y$ is a labelling function (the labels coming from a set $Y$). 
A \emph{path} is a non-empty sequence $e_1,\dots,e_m$ of consecutive edges ($\omega(e_j)=\alpha(e_{j+1})$ for all $1\leq j<m$) or just a vertex $v\in V$ considered
as the empty path around $v$. A path is \emph{simple} if all vertices involved in it, namely $\alpha(e_1),\dots,\alpha(e_m),\omega(e_m)$, are distinct. A path
in which $\omega(e_m)=\alpha(e_1)$ is called a cycle (\emph{based} at $\alpha(e_1)$); it is a \emph{simple cycle} if all initial vertices of its edges are distinct. 
By a \emph{subgraph} of $\Gamma$ we mean a subset of $V\cup E$ closed under $\alpha$ and $\omega$; the subgraph \emph{spanned} by a path consists of all edges of the 
path and all initial and terminal vertices of these edges. For a path $p=e_1\dots e_m$ we define $\alpha(p)=\alpha(e_1)$ and $\omega(p)=\omega(e_m)$. Two paths 
$p$ and $q$ can be composed if $\omega(p)=\alpha(q)$ in which case the resulting path arises as the concatenation of two sequences of edges. Two paths $p,q$ are 
called \emph{coterminal} if $\alpha(p)=\alpha(q)$ and $\omega(p)=\omega(q)$. A (sub)graph is \emph{strongly connected} if for any two vertices $u,v$ involved 
there is a (simple) path $p$ such that $\alpha(p)=u$ and $\omega(p)=v$.

A typical example of a labelled digraph in the sense just presented is the \emph{(right) Cayley graph} $\Cay(G,X)$ of an $X$-generated group $G$. Here the vertices 
are all the elements of the group $G$, the set of labels is $\ol{X}$, and for each $x\in \ol{X}$ there is an edge from $g\in G$ to $g[x]$ labelled by $x$. Notice that every such 
edge $e$ has an \emph{inverse} edge $e^{-1}$ so that $\alpha(e^{-1})=\omega(e)$ and $\omega(e^{-1})=\alpha(e)$, as $g[x][x^{-1}]=g$ must hold for all $g\in G$ and
$x\in \ol{X}$. In such graphs there is a path from $u$ to $v$ if and only if there is a path from $v$ to $u$, and in that sense we may speak simply about the
\emph{connectedness} of graphs and subgraphs. Also, it is customary for the notion of the subgraph to require that it is also closed for the $^{-1}$ operation.
The \emph{label} of a path $p=e_1\dots e_m$ is the word $\ell(e_1)\cdots\ell(e_m)\in \ol{X}^*$ or just the empty word for an empty path around a vertex.
An important remark is that $G$ has a natural action on $\Cay(G,X)$ by labelled graph automorphisms via left multiplication: indeed, for any $g_1,g_2,h\in G$ there is an
edge labelled by $x\in\ol{X}$ from $g_1$ to $g_2$ if and only if there is such an edge from $hg_1$ to $hg_2$. In this sense, for a subgraph $\Delta\subseteq
\Cay(G,X)$ and $g\in G$ we write $g\Delta$ as the result of this left action on $\Delta$: all the vertices of $\Delta$ are multiplied from the left by $g$.
In $\Cay(G,X)$ whenever we spot a cycle based at $1$ and labelled by $w\in\ol{X}^*$ then we have $[w]_G=1$ (for which we also say that ``$w=1$ holds in $G$''); 
however, because of the previous fact, this is true for \emph{any} cycle in $\Cay(G,X)$, regardless on the vertex $g$ it is based at, because if $\Xi$ is the
subgraph spanned by the cycle then $g^{-1}\Xi$ is a subgraph spanned by a cycle based at $1$.

Cayley graphs are \emph{deterministic} in the sense that every path is uniquely determined by its initial vertex and its label. For a word $w\in\ol{X}^*$
we denote by $\Gamma_w$ the subgraph of $\Cay(G,X)$ spanned by the path starting at $1$ and labelled by the word $w$. Note that such a graph necessarily contains
the vertex $[w]$, as this is the terminal vertex of the path inducing $\Gamma_w$. In fact, for the sake of simplifying somewhat the terminology, we are going to
identify simple paths and subgraphs spanned by them whenever this does not cause confusion.

\subsection{Margolis-Meakin expansions} \label{subsec:mgx}

Let $G$ be an $X$-generated group and let $M(G,X)$ be the collection of all pairs $(\Gamma,g)$ such that $\Gamma$ is a finite connected subgraph of $\Cay(G,X)$ 
containing the vertices $1$ and $g$. $M(G,X)$ turns into an inverse monoid with the operations defined by
$$
(\Delta,g)(\Xi,h) = (\Delta\cup g\Xi,gh)
$$
and
$$
(\Delta,g)^{-1} = (g^{-1}\Delta,g^{-1}).
$$
$M(G,X)$ is called the \emph{Margolis-Meakin expansion} of the $X$-generated group $G$. This inverse monoid is $X$-generated with respect to $x\mapsto
(\Gamma_x,[x])$, $x\in X$, and it is the universal $X$-generated $E$-unitary inverse monoid with maximum group image $G$ \cite[Theorem 2.2]{MM89}. 

The natural order on $M(G,X)$ is given by $(\Delta,g)\leq (\Xi,h)$ if and only if $\Xi\subseteq\Delta$ and $g=h$. Also, $(\Delta,g)\mathrel{\sigma} (\Xi,h)$ 
if and only if $g=h$.

By \cite[Corollary 2.9]{MM89} an inverse monoid presentation for $M(G,X)$ over the generators $X$ is given by relations $w^2=w$ for all $w\in\ol{X}^*$ such that $w=1$
holds in $G$. Notice that when $G=FG(X)$ is the free group on $X$ then $w=1$ in $G$ if and only if $\red(w)$ is the empty word (such words are called \emph{Dyck words}).
However, it can be shown by an easy inductive argument that any Dyck word represents an idempotent in any inverse monoid, and hence all the defining relations become
redundant. Therefore, the free inverse monoid $FIM(X)$ is, as discovered earlier in \cite{Munn,Sch}, just the expansion $M(FG(X),X)$ of the free group $FG(X)$ with respect
to the standard set of its free generators.

\subsection{Van Kampen diagrams} \label{subsec:van-k}

\emph{Van Kampen diagrams} \cite{LSch} are convenient topological tools invented to analyse the word problem of a group, that is the question whether
a word $w\in\ol{X}^*$ satisfies $w=1$ in the group. We assume that the underlying group is given by a presentation $G=\Gp\pre{X}{w_i=1\; (i\in I)}$ where all
words $w_i$ are cyclically reduced. A van Kampen diagram $D$ over $G$ is a cell 2-complex along with its embedding into the real plane $\mathbb{R}^2$ such that
the following properties hold:
\begin{itemize}
\item there is a vertex (a 0-cell) called the \emph{base vertex};
\item the underlying 1-skeleton of $D$ is a planar directed $\ol{X}$-labelled graph (so that the 1-cells are directed edges labelled by elements of $\ol{X}$)
closed under $^{-1}$, meaning that for any edge $e$ that belongs to $D$ the edge $e^{-1}$ also belongs to $D$;
\item $D$ is connected and simply connected;
\item for every \emph{region} (2-cell) $R$ of $D$, every vertex in the cycle formed by its \emph{boundary} $\partial R$ and every choice of the two possible directions,
the word labeling the cycle based at the chosen vertex is a cyclic conjugate of either $w_i$ or $w_i^{-1}$ for some $i\in I$.
\end{itemize}

Here, a word $w'$ is a \emph{cyclic conjugate} of a word $w$ if there is a factorisation $w'=uv$ such that $vu=w$. The \emph{area} of $D$ is the number of its regions.

Because of the connectedness and simple connectedness conditions, the boundary $\partial D$ of this complex (which is also the boundary of the unbounded complement
of its planar embedding into $\mathbb{R}^2$) induces a unique cycle (modulo direction) of its 1-skeleton. This cycle is called the \emph{boundary cycle} of the
diagram; if the label of the boundary cycle (upon fixing a direction) is the word $w\in\ol{X}^*$ then we say that we have a \emph{van Kampen diagram for $w$ over $G$}. 

A fundamental result for van Kampen diagrams is the \emph{van Kampen Lemma} \cite[Theorem V.1.1 \& Lemma V.1.2]{LSch} which states (along with some extra information) 
that $w=1$ holds in $G$ if and only if $w$ has a van Kampen diagram over $G$.


\section{The universal inverse monoid with the strong $F$-inverse property} \label{sec:univ}

Let $G$ be an $X$-generated group; throughout we assume that $G$ is given by a presentation $G=\Gp\pre{X}{w_i=1\; (i\in I)}$ where $w_i\in\ol{X}^*$ are cyclically 
reduced words. As it becomes apparent from Subsection \ref{subsec:mgx}, an element $(\Gamma,g)\in M(G,X)$ is maximal in its $\sigma$-class if and only if $\Gamma$
is the graph spanned by a simple path $1\leadsto g$ in $\Cay(G,X)$. Indeed, if $\Gamma$ is not spanned by one such simple paths, then by the connectedness condition
there is a subset $E'\subset E(\Gamma')$ forming a simple path from $1$ to $g$, and then we have $(\Gamma,g)<(\Gamma',g)$ for the graph $\Gamma'$ spanned by $E'$.
Note that in the case $g=1$ there is only one maximal element in the corresponding $\sigma$-class: the one formed by the empty graph (with no edges, just 
the single vertex $1$).

Bearing in mind our interest in the strong $F$-inverse property (defined in the introduction), let $\xi_G$ be the congruence of $M(G,X)$ generated by all pairs of the form
$$
\left( (\Pi_1,g),(\Pi_2,g) \right)
$$
where $g\in G$ and $\Pi_1,\Pi_2$ are two subgraphs of $\Cay(G,X)$ spanned by some simple paths from $1$ to $g$. Define $M_{sF}(G,X)=M(G,X)/\xi_G$. Then the 
following observations become immediate.

\begin{pro}
\begin{enumerate}
\item The inverse monoid $M_{sF}(G,X)$ is $X$-generated and strongly $F$-in\-ver\-se, with maximum group image $G$.
\item $M_{sF}(G,X)$ is universal for the class of all $X$-generated inverse monoids that are strongly $F$-inverse and have a quotient of $G$ as the 
maximum group image. More precisely, if $S$ is an $X$-generated inverse monoid that is strongly $F$-inverse such that there is a canonical group morphism 
$\nu:G\to S/\sigma$, where $\sigma$ is the minimum group congruence of $S$, then there exists a canonical inverse monoid morphism $\phi:M_{sF}(G,X)\to S$ 
such that the following diagram of canonical morphisms commutes:
$$
\xymatrix@C+2pc{M_{sF}(G,X) \ar[r]^{\phi} \ar[d]_{} & S \ar[d]^{} \\
G \ar[r]_{\nu} & S/\sigma}
$$
\end{enumerate}
\end{pro}

\begin{proof}
(1) $M_{sF}(G,X)$ is clearly $X$-generated, via $\iota: x\mapsto (\Gamma_x,[x])\xi_G$. Since $\xi_G$ is contained in the $\sigma$ relation of $M(G,X)$, we
have that $M_{sF}(G,X)/(\sigma/\xi_G)\cong M(G,X)/\sigma \cong G$, so $G$ must be the maximum group image of $M_{sF}(G,X)$ and its minimum group congruence is 
$\sigma/\xi_G$. 
For each $\sigma/\xi_G$-class, the $\xi_G$-class containing all the maximal elements of $M(G,X)$ contained in the corresponding $\sigma$-class
is the maximum element of the $\sigma/\xi_G$-class in question. Indeed, let $(\Pi_1, g)$ be a maximal element of its $\sigma$-class in $M(G,X)$ and 
$(\Delta, g)$ any element of the same $\sigma$-class. We show that $(\Pi_1,g)\xi_G \geq (\Delta,g)\xi_G$. Since $(\Delta,g) \mathrel{\sigma} (\Pi_1,g)$ 
and every element of $M(G,X)$ is below a maximal element, there is a maximal element $(\Pi_2,g)$ of the same $\sigma$-class such that $(\Delta,g)\leq (\Pi_2,g)$. 
This implies that $(\Delta,g)\xi_G\leq (\Pi_2,g)\xi_G = (\Pi_1,g)\xi_G$. Noting that $(\Delta,g)$ runs through the $\sigma$-class of $(\Pi_1,g)$ if and only if 
$(\Delta,g)\xi_G$ runs through the $\sigma/\xi_G$ class of $(\Pi_1,g)\xi_G$, we conclude that $(\Pi_1,g)\xi_G$ is the maximum element in its $\sigma/\xi_G$-class.
Therefore, $M_{sF}(G,X)$ is$F$-inverse, and the very definition of $M_{sF}(G,X)$ shows that it is actually strongly $F$-inverse.

(2) By \cite[Theorem 2.2]{MM89} (see also \cite[Proposition 2.1 (7)]{KLF}), there exists a canonical morphism of inverse monoids $\phi_0:M(G,X)\to S$ such that
the following diagram commutes:
$$
\xymatrix@C+2pc{M(G,X) \ar[r]^{\phi_0} \ar[d]_{} & S \ar[d]^{} \\
G \ar[r]_{\nu} & S/\sigma}
$$
To prove our statement, it suffices to show that $\phi_0$ and the canonical morphism of $M(G,X)$ onto its maximum group image $G$ both factorise through 
$\psi:M(G,X)\to M_{sF}(G,X)$, the canonical quotient map. For the latter of these two morphisms, this follows from the fact that $\xi_G\subseteq\sigma$, 
noted in the previous part. So, it suffices to show that $\ker\psi=\xi_G\subseteq\ker\phi_0$.

Indeed, assume that $\Pi_1$ and $\Pi_2$ are two subgraphs of $\Cay(G,X)$ spanned by simple paths connecting $1$ and $g$; then $(\Pi_1,g)$ and $(\Pi_2,g)$ 
are two maximal elements in a $\sigma$-class of $M(G,X)$. By the previous diagram, $\phi_0(\Pi_1,g)$ and $\phi_0(\Pi_2,g)$ belong to the same class of
the minimum group congruence of $S$. Since $S$ is assumed to be strongly $F$-inverse, they must be both the top element in that class, so we have 
$\phi_0(\Pi_1,g)=\phi_0(\Pi_2,g)$, which was exactly what was to be shown.
\end{proof}

\begin{cor}\label{cor:hom}
Let $S$ be an $X$-generated inverse monoid and let $\sigma$ be its minimum group congruence. Then $S$ is strongly $F$-inverse if and only if $S$ is a homomorphic image of $M_{sF}(S/\sigma,X)$.
\end{cor}

Call a cyclically reduced word $w\in\ol{X}^*$ \emph{cyclic} with respect to $(G,X)$ if $w$ labels a simple cycle in the Cayley graph $\Cay(G,X)$.
(Note that in such a case $w=1$ necessarily holds in $G$.) It is easy to see that $w$ is cyclic if and only if for any proper prefix $u$ of $w$ such that
$w=uv$ we have that the words $u$ and $v^{-1}$ label coterminal simple paths in $\Cay(G,X)$ that have no vertices in common except their initial vertex $1$
and terminal vertex $[u]_G$. 

Let $\theta_G$ be the set of all pairs of words $(u,v)$, where $u,v\in\ol{X}^*$ are non-empty, such that $uv$ is cyclic with respect to $(G,X)$ and let 
$\theta_G^\sharp$ be the congruence on $M(G,X)$ generated by the set
$$
\{\left((\Gamma_u,[u]),(\Gamma_{v^{-1}},[v^{-1}])\right):\ (u,v)\in\theta_G\}.
$$
Here is the main result of this section.

\begin{thm}\label{thm:pre1}
Let $G$ be an $X$-generated group. Then $\xi_G=\theta_G^\sharp$, and thus $M_{sF}(G,X)$ is presented, as an $X$-generated inverse monoid, by the relations:
\begin{itemize}
\item[(i)] $w^2=w$ for all words $w\in\ol{X}^*$ such that $w=1$ holds in $G$,
\item[(ii)] $u=v^{-1}$ for all words $u,v\in\ol{X}^*$ such that $(u,v)\in\theta_G$.
\end{itemize}
\end{thm}

\begin{proof}
The containment $(\supseteq)$ is immediate from the very definition of $\theta_G^\sharp$ as the latter congruence is generated by a set of $\xi_G$-related pairs.
So, we need to prove $(\subseteq)$: for any two coterminal simple paths $\Pi_1$ and $\Pi_2$, connecting $1$ and some element $g\in G$ in $\Cay(G,X)$, we have to prove that 
$$\left((\Pi_1,g),(\Pi_2,g)\right) \in \theta_G^\sharp.$$
Once we show this, the theorem follows straightforwardly, as \cite[Corollary 2.9]{MM89} implies that the relations (i) comprise a presentation for the 
Margolis-Meakin expansion $M(G,X)$.

Assume that $\Pi_1,\Pi_2$ are labelled, respectively, by words $u$ and $v$ (implying that $[u]_G=[v]_G=g$). Let $y_0=1,y_1,\dots,y_{m-1},y_m=g$ be the vertices 
of $\Pi_1$ that also belong to $\Pi_2$ numbered in the order as we traverse the path from $1$ to $g$; this induces a factorisation $u=u_1\cdots u_m$, where $u_j$ 
is the subword of $u$ read in the segment of the path from $y_{j-1}$ to $y_j$. Analogously, let $z_0=1,z_1,\dots,z_{m-1},z_m=g$ be the vertices of $\Pi_2$ belonging 
also to $\Pi_1$ in the order as they occur on $\Pi_2$ from $1$ to $g$, giving rise to a factorisation $v=v_1\cdots v_m$. Clearly, 
$\{y_1,\dots,y_{m-1}\}=\{z_1,\dots,z_{m-1}\}$, however the order in which $y_1,\dots,y_{m-1}$ appear on $\Pi_2$ may be different, so there is a permutation $\pi$ 
of the set $\{0,1,\dots,m-1,m\}$ such that $z_j=y_{\pi(j)}$ for all $0\leq j\leq m$ with $\pi(0)=0$ and $\pi(m)=m$.

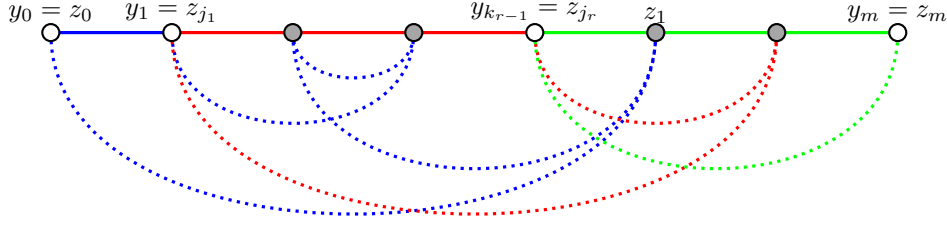
\begin{figure}[t]
\begin{center}{\small
\begin{tikzpicture} [scale=1.6]

\draw[blue, very thick] (0,0) -- (1,0);
\draw[red, very thick] (1,0) -- (4,0);
\draw[green, very thick] (4,0) -- (7,0);

\draw[blue, very thick, dotted] (0,0) .. controls (0,-2) and (5,-2) .. (5,0);
\draw[blue, very thick, dotted] (5,0) .. controls (5,-1.5) and (2,-1.5) .. (2,0);
\draw[blue, very thick, dotted] (2,0) .. controls (2,-0.5) and (3,-0.5) .. (3,0);
\draw[blue, very thick, dotted] (3,0) .. controls (3,-1) and (1,-1) .. (1,0);
\draw[red, very thick, dotted] (1,0) .. controls (1,-2) and (6,-2) .. (6,0);
\draw[red, very thick, dotted] (6,0) .. controls (6,-1) and (4,-1) .. (4,0);
\draw[green, very thick, dotted] (4,0) .. controls (4,-1.5) and (7,-1.5) .. (7,0);

\filldraw[color=black, fill=white, thick] (0,0) circle (2pt) node[anchor=south]{$y_0=z_0$};
\filldraw[color=black, fill=white, thick] (1,0) circle (2pt) node[anchor=south]{$y_1=z_{j_1}$};
\filldraw[color=black, fill=gray!70, thick] (2,0) circle (2pt);
\filldraw[color=black, fill=gray!70, thick] (3,0) circle (2pt);
\filldraw[color=black, fill=white, thick] (4,0) circle (2pt) node[anchor=south]{$y_{k_{r-1}}=z_{j_r}$};
\filldraw[color=black, fill=gray!70, thick] (5,0) circle (2pt) node[anchor=south]{$z_1$};
\filldraw[color=black, fill=gray!70, thick] (6,0) circle (2pt);
\filldraw[color=black, fill=white, thick] (7,0) circle (2pt) node[anchor=south]{$y_m=z_m$};

\end{tikzpicture}
\caption{The full line represents the path $\Pi_1$ and the dotted line represents the path $\Pi_2$. The simple cycles considered in the proof, corresponding to
\eqref{eq:1}, \eqref{eq:2} and \eqref{eq:3}, are coloured in blue, red, and green, respectively. The vertices $y_0,y_1,\dots,y_{k_{r-1}},\dots,y_m$ are highlighted in white. 
In this particular instance, the permutation $\pi^{-1}$ maps $01234567$ to $05231647$.}
\label{fig:paths}
}\end{center}
\end{figure}

Let $j_1=\pi^{-1}(1)$. Then the initial segment of the path $\Pi_1$ from $y_0=1$ to $y_1$ and the initial segment of $\Pi_2$ passing through vertices 
$z_1,z_2,\dots$ up until $z_{j_1}=y_{\pi(j_1)}=y_1$ are two coterminal paths in $\Cay(G,X)$ that have now vertices in common except the initial and the terminal one, 
which implies that either $j_1=1$ and $u_1=v_1$ or $u_1(v_1\cdots v_{j_1})^{-1}$ is a cyclic word. In either case we have that
\begin{equation} \label{eq:1}
(u_1,v_1\cdots v_{j_1}) \in \theta_G.
\end{equation}
Now we are going to define two sequences of indices 
$(j_s)_{s\geq 1}$ and $(k_s)_{s\geq 0}$ starting from $j_1$ and $k_0=1$, respectively. Assume that for some $r\geq 0$
the indices $1\leq j_1<\cdots<j_r$ and $1<k_1\cdots<k_{r-1}$ have already been chosen, labelling the vertices $z_{j_1},\dots,z_{j_r}$ and $y_{k_0},\dots,y_{k_{r-1}}$.
Let
$$
k_r = \min\left(\{1,\dots,m-1\}\setminus\{\pi(1),\pi(2),\dots,\pi(j_r)=k_{r-1}\}\right),
$$
provided that the previous set is not empty. In other words, we are looking for the vertex $y_k$ with the least index having the property that is has not been yet 
visited by the initial segment of the path $\Pi_2$ considered thus far (passing through vertices $z_{j_1},\dots,z_{j_r}$). Then let $j_{r+1} = \pi^{-1}(k_r)$. 
Similarly as before, the segment $\Sigma_1$ of $\Pi_1$ between $y_{k_{r-1}}$ and $y_{k_r}$ and the segment $\Sigma_2$ of $\Pi_2$ between $z_{j_r}=y_{k_{r-1}}$ and 
$z_{j_{r+1}}=y_{\pi(j_{r+1})}=y_{k_r}$ form  a pair of coterminal simple paths. Aside from the initial and the terminal vertex, these two simple paths do not have 
vertices in common by the very choice of $k_r$: all vertices $y_s$ of $\Sigma_1$ strictly between $y_{k_{r-1}}$ and $y_{k_r}$ (if any) have been already visited 
by $\Pi_2$ so none of them can belong to $\Sigma_2$. Hence, we either have that $k_r= k_{r-1}+1$ and $j_{r+1} = j_r+1$ with $u_{k_{r-1}+1} = v_{j_r+1}$ or the word 
$$u_{k_{r-1}+1}\cdots u_{k_r}\left(v_{j_r+1}\cdots v_{j_{r+1}}\right)^{-1}$$
is cyclic. It follows that 
\begin{equation} \label{eq:2}
(u_{k_{r-1}+1}\cdots u_{k_r},v_{j_r+1}\cdots v_{j_{r+1}}) \in \theta_G.
\end{equation}
The other possibility is that we have arrived at 
$$\{\pi(1),\pi(2),\dots,\pi(j_r)\}=\{1,\dots,m-1\},$$ 
whence we must have $j_r=m-1$. Then we choose $k_r=j_{r+1}=m$ and, since either $k_{r-1}=m-1$ with $u_m=v_m$ holds or $u_{k_{r-1}+1}\cdots u_mv_m^{-1}$ is a 
cyclic word, we deduce 
\begin{equation} \label{eq:3}
(u_{k_{r-1}+1}\cdots u_m,v_m) \in \theta_G.
\end{equation}
The equations \eqref{eq:1}, \eqref{eq:2} and \eqref{eq:3} imply
$$
\big((\Pi_1,[u]_G),(\Pi_2,[v]_G)\big)\in \theta_G^\sharp,
$$
as required.
\end{proof}

Bearing in mind the the results of \cite{Nora} (more specifically, Proposition 4.3 of that paper) which shows that there is a bijective correspondence between 
$E$-unitary $X$-generated inverse monoids with maximum group image $G$ and closure operators on the Cayley graph of $G$ with respect to $X$, we now build 
a model for $M_{sF}(G,X)$. 

Call a subgraph $\Delta\subseteq\Cay(G,X)$ \emph{cyclic} if for every edge $e$ of $\Delta$ it contains all edges of any simple cycle containing $e$. 
Then it is straightforward to show that for any subgraph $\Delta\subseteq\Cay(G,X)$ there is the smallest cyclic subgraph of $\Cay(G,X)$ containing $\Delta$, which is called the \emph{cyclic closure} of $\Delta$ and
denoted here by $\Delta^\circ$; also, the map $\Delta\mapsto\Delta^\circ$ turns out to be a $G$-invariant, finitary closure operator on the family of 
all connected subgraphs of $\Cay(G,X)$ in the sense of \cite[Subsection 2.2]{Nora}. The subgraphs of the form $\Delta=F^\circ$ for a finite connected subgraph 
$F\subseteq\Cay(G,X)$ are called \emph{compact}. Consider the set of all pairs of the form
$$
(\Delta,g)
$$
where $g\in G$ and $\Delta$ is a compact connected subgraph of $\Cay(G,X)$ containing the vertices $1$ and $g$. Upon recalling the left action of $G$ on the graph
$\Cay(G,X)$ from the previous section, define the operations
$$
(\Delta,g)(\Xi,h) = ((\Delta\cup g\Xi)^\circ,gh)
$$
and 
$$
(\Delta,g)^{-1} = (g^{-1}\Delta,g^{-1}).
$$
As proved in \cite[Lemma 4.2]{Nora}, in this way we obtain an $X$-generated $E$-unitary inverse monoid $S_\circ$. Note that, by the very definition, if $\Delta,\Xi$ 
are assumed to be cyclic, so is $\Delta\cup g\Xi$, and thus we may omit the application of the closure operator $^\circ$ from the definition of the multiplication.
The inverse monoid generators of $S_\circ$ are the pairs $(\Gamma_x^\circ,[x])$, $x\in X$. Now, the immediate joint effect of \cite[Proposition 4.3]{Nora} and Theorem
\ref{thm:pre1} above is as follows, upon noticing that whenever $w=1$ holds in $G$, the words $w$ and $w^2$ span the same subgraph of $\Cay(G,X)$ (namely the cycle
induced by $w$).

\begin{pro}
Let $G$ be an $X$-generated group. Then $M_{sF}(G,X)\cong S_\circ$.
\end{pro}

\begin{rmk}
The previous proposition and the discussion preceding it reveal how diverse the objects $M_{sF}(G,X)$ can be, depending on the geometry of $\Cay(G,X)$
and, in particular, the way in which cycles overlap in this graph. Let us take a quick look at two constrasting examples. 
First, in the Cayley graph of the free $X$-generated group $G=FG(X)$ there are no cycles at all (as it is a tree), 
so the cyclic closure of any connected finite subgraph is the subgraph itself. Hence, $M_{sF}(FG(X),X)$ is just $M(FG(X),X)$,
which can be also easily seen from Theorem \ref{thm:pre1}. On the other hand, consider the case $X=\{a,b\}$ and $G=\Gp\pre{a,b}{a^{-1}b^{-1}ab=1}$, the free
Abelian group on two generators. The Cayley graph is an infinite grid in $\mathbb{R}^2$ whose elements can be thought of a pairs of integers, made up from squares 
(4-cycles) formed by two coterminal paths labelled by $ab$ and $ba$. It is readily seen that the cyclic closure of \emph{any} non-empty subgraph 
is the whole Cayley graph $\Cay(G,X)$. Therefore, with the exception of the element $\lambda=(\es,1)$ (where $\es$ denotes the graph with one vertex, $1$, and no edges), 
all elements with the same group component are identified by $\xi_G$, and we have $M_{sF}(G,X)\cong G\cup\{\lambda\}$ (where now $\lambda$ acts
as an externally added identity element to the group $G$). 

The group $G=\Gp\pre{a,b}{(ab)^2=1}$ provides an intermediate example ($G\cong \mathbb{Z}\ast\mathbb{Z}_2$ with respect to $a$ and $t=ab$). Namely, $\Cay(G,\{a,b\})$ 
is here a ``tree of squares'' (4-cycles labelled by $abab$), and at every vertex two such cycles meet (one containing the outgoing edge $a$ and incoming edge $b$, 
the other containing the outgoing edge $b$ and incoming edge $a$) but without any edge overlaps. So, the cyclic closure of a single edge is the unique 4-cycle it belongs to, 
and thus for any subgraph $\Gamma$, $\Gamma^\circ$ is the union of the 4-cycles (``squares'') containing the edges of $\Gamma$.
\end{rmk}


\section{An application to one-relator special inverse monoids} \label{sec:sims}

The goal of this section is to apply the general result from the previous section to obtain a description of one-relator special inverse monoids $\Inv\pre{X}{w=1}$ 
with the strong $F$-inverse property, where $w$ is a cyclically reduced word. This characterisation will be given in terms of the factorisation of the relator word 
$w$ into minimal invertible pieces (cf.\ Subsection \ref{subsec:inv}).

We say that an $X$-generated group $G$ has the \emph{property $(\dagger)$} if it admits a presentation $G=\Gp\pre{X}{w_i=1\; (i\in I)}$ such that each relator word $w_i$ 
is cyclic. This is equivalent to requiring that no subword of any of the words $w_i$ equals $1$ in $G$. If this is indeed the case, all words $w_i$ must be cyclically 
reduced. By a result of Weinbaum \cite{W2}, all one-relator groups have this property, and this is witnessed by \emph{any} of their presentations by one defining relation 
involving a cyclically reduced word. Earlier, the same author proved \cite{W1} that the same property is enjoyed by groups satisfying the small cancellation condition 
$C'(1/6)$ (the so-called \emph{sixth groups}), see also \cite{LSch}. With such a property of the maximum group image of an inverse monoid at hand, the presentation from 
Theorem \ref{thm:pre1} can be considerably simplified.

\begin{thm}\label{thm:pre2}
Let $G$ be an $X$-generated group with the property $(\dagger)$ and let 
$$G=\Gp\pre{X}{w_i=1\; (i\in I)}$$ 
be any of its presentations witnessing this property. Then the inverse monoid $M_{sF}(G,X)$ is presented by the relations 
$$u=v^{-1}$$ 
for all non-empty words $u,v\in\ol{X}^*$ such that $uv$ is a cyclic conjugate of $w_i$ for some $i\in I$.
\end{thm}

\begin{proof}
Bearing in mind Theorem \ref{thm:pre1}, our aim is to show that the relations:
\begin{itemize}
\item[(A)] $u=v^{-1}$ for all words $u,v\in\ol{X}^*$ such that $(u,v)\in\theta_G$, and
\item[(B)] $w^2=w$ for all words $w\in\ol{X}^*$ such that $w=1$ holds in $G$, 
\end{itemize}
all follow (in the context of inverse monoid presentations) from the relations given in the formulation of the theorem.

(A) Let $w\in\ol{X}^*$ be a word such that $w=1$ holds in $G$ and $w$ induces a simple cycle in $\Cay(G,X)$. Let $w=uv$ be an arbitrary factorisation of $w$ into non-empty 
words. By the van Kampen Lemma, there exists a van Kampen diagram $D$ such that the label of the boundary cycle $\partial D$ is precisely $w$. By the given conditions,
it follows that all vertices on $\partial D$ must be distinct (of course, aside from the base vertex, which must appear at least and hence exactly twice); for otherwise, 
if some vertex is repeated on $\partial D$ such that at least one of the repetitions involves an intermediate vertex of the cycle, this would give rise to a proper subword
$w'$ of $w$ and a subcycle $C$ of $\partial D$ labelled by $w'$ such that $C$ together with its interior forms a van Kampen diagram $D'\subseteq D$ for $w'$. In such a case 
we would have $\partial D'=C$ and so conclude that $w'=1$ holds $G$, contrary to the  assumption that $w$ induces a simple cycle in $\Cay(G,X)$. Therefore, bearing in mind that $D$ 
is connected and simply connected, the union of its cells embedded into $\mathbb{R}^2$ is homeomorphic to a disc, i.e.\ $D$ must be a disc diagram. 
It follows that $D$ consists entirely of regions and their boundaries (edges and vertices), and each region is labelled by a cyclic conjugate of some $w_i$, $i\in I$, 
in one of the two possible directions (or we may fix a direction and allow a cyclic conjugate of $w_i^{-1}$ to be a label of the boundary of a region).

Now we prove a claim from which the required result follows by applying it to our word $w$. 

\begin{claim}
If $q\in\ol{X}^*$ is a (not necessarily reduced) word which admits a disc van Kampen diagram $D$ with respect to the given presentation of $G$ and $q=uv$ is an arbitrary 
factorisation into non-empty words, then $u=v^{-1}$ can be deduced from the defining relations from the formulation of the theorem. 
\end{claim}

\begin{proof}[Proof of Claim.]
We prove this by induction on the area of $D$.

If $q$ admits a disc diagram of area 1, this means that either $q$ of $q^{-1}$ is a cyclic conjugate of some word $w_i$, $i\in I$; therefore, since $w_i$ defines, by
assumption, a simple cycle in $\Cay(G,X)$, it follows that either $u=v^{-1}$ or $v=u^{-1}$ is one of the relations from the formulation of the theorem. Hence, assume now that $q$ 
admits a disc diagram of area $n>1$, and that every word from $\ol{X}^*$ admitting a disc diagram of area $<n$ the claim holds.

Next we claim that any disc diagram $D_0$ has a region $R'$ with the property that $\partial R' \cap \partial D_0$ is a single non-empty path; in such a case, removing
$R'$ from $D_0$ results in a diagram $D_0'$ that is again a disc diagram (of area one less that $D_0$), where the edges in $\partial R' \cap \partial D_0$ are replaced
by the edges from $\partial R' \setminus \partial D_0$ as the new part of the boundary of $D_0'$. Indeed, pick any edge from $\partial D_0$ and let $R_0$ be the unique
region whose boundary contains that edge. If $R_0$ has the required property we are done; otherwise, the removal of $R_0$ splits $D_0$ into a finite number of disc
diagrams such that the boundary of each of these diagrams consists of a path contained in the original boundary $\partial D_0$ and a piece of $\partial R_0$. Then
pick $D_1$ to be any of these diagrams, and let $R_1$ be any region of $D_1$ containing an edge from $\partial D_0$. If $\partial R_1\cap \partial D_0$ is a single 
non-empty path, we are again done, for $\partial R_1\setminus \partial D_0$ is then entirely contained in the interior od $D_0$. Otherwise, the removal of $R_1$  
splits $D_1$ into finitely many disc diagrams at least one of which has the property that its boundary consists of a non-empty path from $\partial D_0$ and a piece of
$\partial R_1$, as $\partial R_1\cap \partial D_0$ is a disjoint union of at least two (not necessarily non-empty) paths. Upon letting $D_2$ to be one such diagram,
we repeat the construction. However, this process cannot be repeated indefinitely as the areas of $D_0,D_1,D_2,\dots$ form a strictly decreasing sequence of positive
integers. Therefore, we eventually must arrive to a region with the required property.

Returning to the original claim, let $y_0$ be the base point of $\partial D$ and let $R$ be a region of $D$ such that $\partial R\cap \partial D$ is a single non-empty path
$\Pi$. Also, let $z\in\partial D$ be the the vertex determined by the factorisation $q=uv$. 

Consider first the case when $\Pi$ does not contain $y_0$. Let $q_1$ be the word labelling $\Pi$ and let $q_1q_2$ be the label of $\partial R$ (in the orientation
determined by $\Pi$). Then $q_1q_2$ is a cyclic conjugate of one of the words $w_i$ or $w_i^{-1}$, so $q_1=q_2^{-1}$ is one of the available relations from the formulation 
of the theorem. On the other hand, if $D'$ is the disc diagram obtained by deleting $R$ from $D$, then $D'$ has area $n-1$ and the label of $\partial D'$ is obtained from $q$ 
by replacing the occurrence of its subword $q_1$ corresponding to $\partial R\cap \partial D$ by $q_2^{-1}$. If $z$ is situated on $\partial D$ after we have traversed the boundary 
of $R$, this occurrence of $q_1$ is entirely contained in $u$, so if we replace it by $q_2^{-1}$ we obtain a word $u'$ such that $u=u'$ can be deduced from the relations from the 
formulation of the theorem. By the induction hypothesis, however, $u'=v^{-1}$ can be deduced as well, thus proving the claim. A similar argument leads us to the same conclusion
when $z$ lies on $\partial D$ before it traverses the boundary of $R$. So, it remains to analyse the case when $z\in\partial R$. Then $q=uv=(u_0u_1)(v_1v_0)$, where $u_1v_1=q_1$,
and so from the region $R$ and the available relations related to its boundary we have
$$
u_1 = (v_1q_2)^{-1} = q_2^{-1}v_1^{-1}.
$$
By the induction hypothesis, from the disc diagram $D'$ we have
$$
u_0q_2^{-1} = v_0^{-1},
$$
yileding 
$$
u = u_0u_1 = u_0q_2^{-1}v_1^{-1} = v_0^{-1}v_1^{-1} = (v_1v_0)^{-1} = v^{-1}.
$$

Now assume that the base point $y_0$ belongs to $\Pi=\partial R \cap \partial D$. Let $y=\omega(\Pi)$ and $y'=\alpha(\Pi)$ so that the subpath of $\Pi$ between $y_0$ and $y$ is 
the longest initial segment of $\partial D$ contained in $\partial R$ (corresponding to the prefix $q_0$ of $q$) and the subpath of $\Pi$ between $y'$ and $y_0$ is the longest 
final segment of $\partial D$ contained in $\partial R$ (corresponding to the suffix $q_\infty$ of $q$).

Assume first that $z$ is located on $\partial D$ between $y$ and $y'$ (including the possibility that it coincides with one of these vertices), so that $z$ is \emph{not} an
interior vertex of $\Pi$. Let $u_0$ be the subword of $q$ (and indeed a suffix of $u$) labelling the segment of $\partial D$ between $y$ and $z$, and let $v_0$ be the prefix 
of $v$ labelling the segment of $\partial D$ between $z$ and $y'$. Since the word $q_0tq_\infty$ labelling the boundary $\partial R$ is a cyclic conjugate of one of the words 
$w_i$ or $w_i^{-1}$ (where $t$ labels the part of that boundary) contained in the interior of $D$), we have that the relations from the formulation of the theorem imply 
$$
q_0 = (tq_\infty)^{-1}.
$$ 
If we delete the region $R$ from the diagram and move the base point to $y'$ (say), we get a disc diagram $D''$ of area $n-1$ such that $\partial D''$ is labelled by 
$t^{-1}u_0v_0$. By the inductive hypothesis, 
$$
t^{-1}u_0 = v_0^{-1}
$$
follows from  relations from the formulation of the theorem, and so
$$
u = q_0u_0 = q_\infty^{-1}t^{-1}u_0 = q_\infty^{-1}v_0^{-1} = (v_0q_\infty)^{-1} = v^{-1}.
$$

Now suppose that $z$ is located on $\partial D$ (strictly) between $y_0$ and $y$, inducing the factorisation $q_0=uv_1$. From the region $R$ and the property of the word
labelling its boundary we obtain
$$
u = (v_1tq_\infty)^{-1}.
$$
In such a case, $v_1$ is a prefix of $v_0$, so $v_0=v_1v_2$ for some word $v_2$. However, again by removing $R$ from the diagram similarly as before and by invoking the 
inductive hypothesis, we have $t^{-1}={v_2^{-1}}$ from which we have $t=v_2$ and so 
$$
u = (v_1v_2q_\infty)^{-1} = (v_0q_\infty)^{-1} =v^{-1}.
$$
The case when $z$ is located on $\partial D$ (strictly) between $y'$ and $y_0$ is analogous to the previous case. Thus the inductive proof of the claim is complete.
\end{proof}

(B) We prove this by induction on the area of the van Kampen diagram $D$ for $w$ over $G$ and its given presentation; note that here we do not assume that $D$ is a disc
diagram, as $w$ is an arbitrary word such that $w=1$ holds in $G$. If the area is $0$ then the diagram in question is in fact a tree and so the word read at the boundary 
$\partial D$ is a Dyck word.  It is a routine exercise to show that in any inverse monoid/semigroup a Dyck word over the set of its generators must represent an idempotent. 
Now assume that the area of $D$ is positive. Pick any region $R$, containing at least one edge $e$ from the boundary of $D$. Let $x\in\ol{X}$ be the label of $e$, and
assume that $u$ is the word labelling the remainder of the boundary of $R$. Then $x=u^{-1}$ is one of the relations from the formulation of the theorem. Now delete
the edge $e$ from $D$. This operation merges the interior of the region $R$ with the unbounded region of the diagram and yields a new diagram $D'$ whose boundary word $w'$
is obtained from $w$ by replacing the occurrence of the letter $x$ labelling the removed edge by the word $u^{-1}$. Hence, $w=w'$ is a consequence of the relations
given in the formulation of the theorem. However, the area of $D'$ is by one smaller than the area of $D$, and so by the induction hypothesis we can derive $(w')^2=w'$
from the given relations. Thus we deduce $w^2=w$ and our inductive proof is complete.
\end{proof}

\begin{cor}\label{cor:one-rel}
Let $w\in\ol{X}^*$ be a cyclically reduced word and $M=\Inv\pre{X}{w=1}$. Then $M$ is strongly $F$-inverse if and only if it satisfies the relations $u=v^{-1}$ 
for all non-empty words $u,v\in\ol{X}^*$ such that $uv$ is a cyclic conjugate of $w$. 
\end{cor}

\begin{proof} 
First note that by \cite[Theorem 4.1]{IMM} the inverse monoid $M$ is $E$-unitary. Since its maximum group quotient is $G=\Gp\pre{X}{w=1}$, it is a quotient of $M(G,X)$. It is thus strongly $F$-inverse if and only if it satisfies all the defining relations of $M_{sF}(G,X)$. The result now follows from the previous theorem aided by 
\cite[Theorem 2]{W2}.
\end{proof}

Still working under the assumption that the maximum group image of a special inverse monoid has the property $(\dagger)$, we have the following more ``operative'' 
criterion that helps us in determining whether the inverse monoid in question has the strong $F$-inverse property. Namely, let $w=a_1a_2\cdots a_n$ be a word 
over $\ol{X}$ and let $M$ be an $X$-generated inverse monoid. For a pair of adjacent letters $a_j,a_{j+1}\in\ol{X}$ we say that it is \emph{matching} with 
respect to $M$ if 
$$r(a_j) = a_j^{-1}a_j = a_{j+1}a_{j+1}^{-1} = d(a_{j+1})$$ 
holds in $M$. (Upon recalling the facts from Subsection \ref{subsec:basic}, the meaning of this condition is that, under the embedding of $M$ into an appropriate 
symmetric inverse monoid, partial injection representing $a_j$ has range that is precisely the same as the domain of that representing $a_{j+1}$.) 
The word $w$ is \emph{linked} with respect to $M$ if every pair of its adjacent letters is matching with respect to $M$.

\begin{pro}\label{pro:linked}
Let $G=\Gp\pre{X}{w_i=1\; (i\in I)}$ be a presentation of the group $G$ with the property that each of its relator words $w_i$ is cyclic. Then the special 
inverse monoid $M=\Inv\pre{X}{w_i=1\; (i\in I)}$ satisfies all the relations from Theorem \ref{thm:pre2} if and only if each word $w_i$, $i\in I$, is linked 
with respect to $M$.

In particular, a one-relator special inverse monoid $M=\Inv\pre{X}{w=1}$, where $w$ is cyclically reduced, is strongly $F$-inverse if and only if $w$ is linked 
with respect to $M$.
\end{pro}

\begin{proof}
($\Rightarrow$)
Assume that for some $i\in I$ we have $w_i=ua_ja_{j+1}v$ where some of the words $u,v\in\ol{X}^*$ may be empty and $a_j,a_{j+1}\in\ol{X}$. Since $w_i$ is by 
assumption cyclic in $G$, it follows that in $M$ we must have
$$
a_ja_{j+1} = (vu)^{-1}.
$$
For the same reason, we have that
\begin{align*}
a_j &= (a_{j+1}vu)^{-1} = (vu)^{-1}a_{j+1}^{-1} = a_ja_{j+1}a_{j+1}^{-1}, \\
a_{j+1} &= (vua_j)^{-1} = a_j^{-1}(vu)^{-1} = a_j^{-1}a_ja_{j+1}
\end{align*}
also holds in $M$. Therefore, $a_j^{-1}a_j = a_j^{-1}a_ja_{j+1}a_{j+1}^{-1} = a_{j+1}a_{j+1}^{-1}$.

($\Leftarrow$) For this direction and for later use, it is convenient to record the following
\begin{claim}
If the word $w=a_1a_2\cdots a_n$ is linked in an inverse monoid $M$ and $z=a_j\cdots a_k$ is a subword of $w$, $1\leq j\leq k\leq n$, then $zz^{-1}=a_ja_j^{-1}$ 
and $z^{-1}z = a_k^{-1}a_k$ holds in $M$. 
\end{claim}

\begin{proof}[Proof of Claim.]
We have
$$
zz^{-1} = a_j\cdots a_k(a_j\cdots a_k)^{-1} = a_j\cdots a_ka_k^{-1}\cdots a_j^{-1}.
$$
By the linked condition, $a_ka_k^{-1} = a_{k-1}^{-1}a_{k-1}$ holds in $M$, which yields
$$
zz^{-1} = a_j\cdots a_{k-1}\left(a_{k-1}^{-1}a_{k-1}\right)a_{k-1}^{-1}\cdots a_j^{-1} = a_j\cdots a_{k-1}a_{k-1}^{-1}\cdots a_j^{-1}.
$$
Proceeding inductively, we arrive at $zz^{-1}=a_ja_j^{-1}$, as required. The other equality is proved analogously.
\end{proof}

Now let us assume that $w_i=uzv$ for some words $u,v,z\in\ol{X}^*$ (of which some of $u,v$ may be empty; we assume here they are not, the  ``degenerate'' cases 
being treated similarly). Let $a,b$ be the first and the last letter of $z$, respectively, let $c$ be the last letter of $u$ and $d$ the first letter of $v$. 
Then $z=az'$ for some word $z'$ and so 
$$
z=az'=aa^{-1}az'=c^{-1}caz'=c^{-1}cz
$$
holds in $M$ by the assumed linked condition for $w_i$. However, the previous claim implies that $c^{-1}c=u^{-1}u$ holds in $M$, so we deduce $z=u^{-1}uz$. 
Analogously, we must have $z=zvv^{-1}$ in $M$, implying
$$
z = u^{-1}uzvv^{-1} = u^{-1}v^{-1} = (vu)^{-1},
$$
as $w_i=uzv=1$ holds in $M$. Thus the required conclusions follow, with the one about the one-relator case being a consequence of Corollary \ref{cor:one-rel}.
\end{proof}

\begin{cor}
 Let $|X|=1$ and $w\in\ol{X}^*$ be a non-empty cyclically reduced word. If $M=\Inv\pre{X}{w=1}$ is strongly $F$-inverse then $M$ is a cyclic group.
\end{cor}

\begin{proof}
Let $X=\{a\}$. Without loss of generality, we assume that $a$ is the first letter of $w$. Since $w$ is reduced, $w=a^k$ for some $k$. If $k\geq 2$, the linked condition 
implies that $a^{-1}a = aa^{-1} =1$, that is, $a$ is invertible and $M$ is a cyclic group (and if $k=1$, $M$ is a trivial group). 
\end{proof}

From now on we assume that $|X|\geq 2$.

\begin{lem}
Let $w\in\ol{X}^*$ be a cyclically reduced word and $M=\Inv\pre{X}{w=1}$. Assume that $M$ is strongly $F$-inverse and let $w=w_1\cdots w_k$ be the decomposition 
of $w$ into minimal invertible pieces.
\begin{enumerate}
\item If the letter $a\in\ol{X}$ is the first letter of some piece then each occurrence of $a$ in $w$ is the first letter of some piece, and each occurrence  
of $a^{-1}$ in $w$ is the last letter of some piece.
\item If the letter $a\in\ol{X}$ is the last letter of some piece then each occurrence of $a$ in $w$ is the last letter of some piece, and each occurrence 
of $a^{-1}$ in $w$ is the first letter of some piece.
\item If the letter $a\in\ol{X}$ is a middle letter of some piece then each occurrence of both $a$ and $a^{-1}$ in $w$ is a middle letter of some piece.
\end{enumerate}
\end{lem}

\begin{proof}
Assume $a$ has an occurrence in $w$ as the first letter of some piece and, at the same time, there is a piece of the form $uav$ where $u$ is not empty.
Since $M$ is strongly $F$-inverse, by Proposition \ref{pro:linked} and the claim contained in its proof it follows that $u^{-1}u=aa^{-1}$ holds in $M$. 
On the other hand, since both $a$ and $u$ are prefixes of an invertible piece, we have that $aa^{-1}=uu^{-1}=1$ holds in $M$, so we deduce that $u$ itself 
represents and invertible element of $M$. This is a contradiction with the minimality of the piece $uav$; similarly we get a contradiction if $a^{-1}$
occurs in a position that is not the last letter of some piece. The assertion (2) is dual to (1), and (3) follows from (1) and (2).
\end{proof}

\begin{lem}\label{lem:piece2}
Let $w\in\ol{X}^*$ be a cyclically reduced word and $M=\Inv\pre{X}{w=1}$. Assume that $M$ is strongly $F$-inverse and let $w=w_1\cdots w_k$ be the 
decomposition of $w$ into minimal invertible pieces. Then every invertible piece contains at most two letters.
\end{lem}

\begin{proof}
Seeking a contradiction, assume that there is a piece $w_i$ containing at least three letters, $w_i=aub$ for some non-empty word $u$ and $a,b\in\ol{X}$. By the
previous lemma, $u$ can contain neither of the letters $a,a^{-1},b,b^{-1}$; let $c$ be an arbitrary letter occurring in $u$. We define a map $\psi:\ol{X}\to
\{a,a^{-1},c,c^{-1},1\}$ as follows:
\begin{itemize}
\item if $x\in\ol{X}$ occurs as the first letter in some invertible piece, set $\psi(x)=a$ and $\psi(x^{-1})=a^{-1}$;
\item if $x\in\ol{X}$ occurs as the last letter in some invertible piece, set $\psi(x)=a^{-1}$ and $\psi(x^{-1})=a$;
\item if $x\in\ol{X}$ occurs as a middle letter of some invertible piece, set $\psi(x)=c$ and $\psi(x^{-1})=c^{-1}$;
\item if $x\in\ol{X}$ is such that neither $x$ nor $x^{-1}$ occurs in $w$, set $\psi(x)=\psi(x^{-1})=1$.
\end{itemize}
By the previous lemma, this map is well defined (i.e.\ the above definition is consistent) and total. Hence, $\psi$ can be extended to a homomorphism of 
free inverse monoids $\ol\psi: FIM(X) \to FIM(\{a,c\})$ (and in fact to a homomorphism of free involution monoids $\ol{X}^*\to\{a,a^{-1},c,c^{-1}\}^*$
which we will by a slight abuse of notation also denote by $\ol\psi$), and also there is a homomorphism $\widehat\psi: M \to M_\psi=\Inv\pre{a,c}{\ol\psi(w)=1}$ 
such that the diagram
$$
\xymatrix@C+2pc{FIM(X) \ar[r]^{\ol\psi} \ar[d]_{\nu} & FIM(\{a,c\}) \ar[d]^{\nu'} \\
M \ar[r]_{\widehat\psi} & M_\psi}
$$
commutes, where $\nu,\nu'$ are the corresponding canonical homomorphisms. Since $M$ is assumed to be strongly $F$-inverse, the word $w$ is linked with respect 
to $M$, so $\ol\psi(w)$ must be linked with respect to $M_\psi$ as the assumption that $x^{-1}x=yy^{-1}$ holds in $M$ for any adjacent pair of letters 
$x,y\in\ol{X}$ in $w$ implies that $\psi(x)^{-1}\psi(x)=\ol\psi(x^{-1}x)=\ol\psi(yy^{-1})=\psi(y)\psi(y)^{-1}$ holds in $M_\psi$ and, by definition, no letter 
appearing in $w$ is mapped to the empty word (thus $\psi(x)$ and $\psi(y)$ are adjacent in $\ol\psi(w)$). Note that $\ol\psi(w)$ is in general not 
(cyclically) reduced as for all $1\leq i\leq k$, $\psi(w_i)$ has the form $aza^{-1}$ for some non-empty word $z$ over $\{c,c^{-1}\}$. So, either 
$a,c$ or $a,c^{-1}$ are adjacent in $\ol\psi(w)$.

Now consider two partial injections from $\mathcal{I}_\N$: let $\tilde{a}(n)=2n$ for all $n\in\N$, and let $\tilde{c}$ be the identity map on $2\N\cup\{1\}$.
Let $S$ be the submonoid of $\mathcal{I}_\N$ generated by these two maps. Then for all $1\leq i\leq k$, the substitution $a\mapsto\tilde{a}$, $c\mapsto\tilde{c}$,
takes every word of the form $\psi(w_i)$ to the identity map on $\N$, and thus $S$ satisfies $\psi(w)=1$ with respect to the specified choice of generators.
Therefore, the previous substitution extends to a surjective homomorphism $\varphi:M_\psi\to S$. However, $\tilde{a}$ does not match with $\tilde{c}=
\tilde{c}^{-1}$, because $r(\tilde{a})$ is the identity map on $2\N$ while $d(\tilde{c})$ coincides with $\tilde{c}$, the identity map on $2\N\cup\{1\}$.
This is a contradiction with the assumption that the word $\ol\psi(w)$ is linked and the existence of the homomorphism $\varphi$; hence, the lemma follows.
\end{proof}

Now we are in a position to state and prove the announced characterisation of strongly $F$-inverse one-relator special inverse monoids with a cyclically 
reduced relator word.

\begin{thm}\label{thm:char}
Let $w\in\ol{X}^*$ be a cyclically reduced word and $M=\Inv\pre{X}{w=1}$. The special inverse monoid $M$ is strongly $F$-inverse if and only if every piece 
$w_i$ in the decomposition $w=w_1\cdots w_k$ of $w$ into minimal invertible pieces has length at most two.
\end{thm}

\begin{proof}
($\Rightarrow$) This is proved in the previous lemma.

($\Leftarrow$) For each $1\leq i\leq k$ let $a_i,b_i\in\ol{X}$ be the first and the last letter, respectively, of $w_i$ (if $|w_i|=1$ we set $a_i=b_i$).
Then every $a_i$ represents a right invertible element of $M$ and every $b_i$ represents a left invertible element of $M$. In other words, $a_ia_i^{-1}=
b_i^{-1}b_i=1$. So, to prove that $w$ is linked (which, by Proposition \ref{pro:linked}, suffices to show the strong $F$-inverse property of $M$), we need 
to argue that $a_i^{-1}a_i=b_ib_i^{-1}$ holds for all $i$ such that $|w_i|=2$. 

Indeed, then we have $a_i\,\R\,1\,\L\,b_i$ and $a_ib_i\,(\R\cap\L)\,1$, so there exist elements $p,q\in M$ such that $a_ib_ip=a_i$ and $qa_ib_i=b_i$. But then
$$b_ipb_i=(qa_ib_i)pb_i=q(a_ib_ip)b_i=qa_ib_i=b_i,$$ 
implying that $b_ip$ is an idempotent. This yields 
$$b_ip=b_ib_i^{-1}b_ip=b_ipb_ib_i^{-1}=b_ib_i^{-1},$$
so $a_ib_ib_i^{-1}=a_i$. Similarly, we obtain $a_i^{-1}a_ib_i=b_i$, and thus $a_i^{-1}a_i =  a_i^{-1}a_ib_ib_i^{-1} = b_ib_i^{-1}$, as required.
\end{proof}

\begin{exa}
\begin{enumerate}
\item If an $X$-generated one-relator inverse monoid $S$ is a group (for example, $S=\Inv\pre{a,b}{aba=1}$) then it is obviously strongly $F$-inverse, as then 
$S$ coincides with the group given by the same presentation, and the canonical map $M(S,X)\to S$ is in fact the homomorphism of $M(S,X)$ onto its greatest group
image (collapsing entire $\sigma$-classes into a single element). This happens if and only if each of the minimal invertible pieces of $w$ has precisely one letter.
\item Let $M=\Inv\pre{a,b}{(ab)^n=1}$.
We show that the minimal invertible pieces of $w=(ab)^n$ with respect to $M$ are precisely the factors $ab$. Since $w=1$, the word $ab$ represents an invertible
element of $M$. We show that $a$ and $b$ are not invertible; for this it suffices to find a homomorphic image of $M$ in which the images of $a$ and $b$ are 
not invertible. Let $B=\Inv\pre{c}{cc^{-1}=1}$ be the bicyclic monoid \cite{How,Law,Pet}. In the inverse monoid $B\times \mathbb{Z}_n$, where 
$\mathbb{Z}_n=\langle d\rangle$ is the cyclic group of order $n$, define $\tilde{a} = (c,d)$ and $\tilde{b} = (c^{-1},1)$. Then we have $(\tilde{a}\tilde{b})^n 
= (1,d)^n = (1,1)$. It follows that the assignment $a\mapsto \tilde{a}$, $b\mapsto \tilde{b}$ extends to a homomorphism $M\to B\times \mathbb{Z}_n$. Since 
$\tilde{a}$ is not invertible in $B\times\mathbb{Z}_n$, $a$ is not invertible in $M$. Hence, $M$ is not a group; however it is strongly $F$-inverse. 
Its greatest group image is the free product $\mathbb{Z}\ast\mathbb{Z}_n$ (with respect to $a$ and $t=ab$).
\end{enumerate}
\end{exa}


\section{Some further examples and non-examples}

We begin this closing section by the observation that a number of relations from Corollary \ref{cor:one-rel} required for a one-relator inverse monoid with 
a cyclically reduced relator word to be strongly $F$-inverse already hold in any such one-relator inverse monoid, independently of the presence of the 
property in question. This pinpoints the difference between the $F$-inverse property and its strong version.

\begin{lem}\label{lem:ineq}
Let $M=\Inv\pre{X}{w=1}$ where $w\in\ol{X}^*$ is a cyclically reduced word. Assume that $w=puq$ where some of the words $p,q$ might be empty. Then
$$
u\geq (qp)^{-1}
$$
holds in $M$. Furthermore, unless $u$ is a ``middle subword'' of some invertible piece of $w$, that is if $w_i=v'uv''$ for some invertible piece $w_i$ of $w$
and some non-empty words $v',v''$, we actually have $u = (qp)^{-1}$ holding in $M$.
\end{lem}

\begin{proof}
The relation $w=1$ implies $qpuqp=qp$. Hence, $(uqp)^2=uqp$ is an idempotent, and so
$$
qp = qpuu^{-1}uqp = qp(uu^{-1})(uqp) = qpuqp(uu^{-1}) = (qp)(u^{-1})^{-1}u^{-1},
$$
yielding $qp\leq u^{-1}$, i.e.\ $(qp)^{-1}\leq u$.

For the second part of the lemma, assume that $w=w_1\cdots w_i\cdots w_k$ is the finest factorisation of $w$ into invertible pieces, and let first $u$ be 
a prefix of $w_i$, $w_i=uv$. Then it follows from \cite[Proposition 4.2]{IMM} that $w_i\cdots w_nw_1\cdots w_{i-1}=1$ holds in $M$. Upon denoting
$u'=qp=(vw_{i+1}\cdots w_k)(w_1\cdots w_{i-1})$ we conclude that $uu'=1$ and holds in $M$ and thus $u'uu'=u'$ and $uu'u=u$. Therefore, $u=(u')^{-1}$.
An analogous conclusion follows when $u$ is a suffix of some invertible piece.

Finally, assume that $u=u'w_i\cdots w_ju''$ where $u'$ is a suffix of a piece, $u''$ is a prefix of a piece, and the product in the middle might be empty.
Then we already know that in $M$ we have $$(u')^{-1}=w_i\cdots w_ju''qp \quad  {\text{and}} \quad (u'')^{-1}=qpu'w_i\cdots w_j. 
$$ Thus we obtain
\begin{align*}
u^{-1} &= (u'')^{-1}(w_i\cdots w_j)^{-1}(u')^{-1} \\ 
& = qpu'(w_i\cdots w_j)(w_i\cdots w_j)^{-1}(w_i\cdots w_j)u''qp\\
& = q(pu'w_i\cdots w_ju''q)p = qp
\end{align*}
bearing in mind that $pu'w_i\cdots w_ju''q= puq = w_1\cdots w_k = 1$.
\end{proof}

Note that the previous lemma provides an alternative way to establish the reverse implication in Theorem \ref{thm:char}, as in the case when each invertible piece 
is of length $\leq 2$ there  are no ``middle subwords'' of pieces at all.

\begin{exa}
There exist one-relator special inverse monoids with a cyclically reduced relator word that are $F$-inverse, but not strongly. For example, consider 
$M=\Inv\pre{a,b,c}{abc=1}$. From the previous lemma it follows that in $M$ we have $a^{-1}=bc$, $c^{-1}=ab$ and $b\geq (ca)^{-1}$ (and so $b^{-1}\geq ca$). 
Also, denoting by $\tilde{a}$ and $\tilde{c}$ the same partial maps from $\mathcal{I}_{\mathbb{N}}$ as in the proof of Lemma \ref{lem:piece2}, the map 
$a\mapsto\tilde{a}$, $b\mapsto\tilde{c}$, $c\mapsto\tilde{a}^{-1}$, extends to a homomorphism from $M$ onto the inverse submonoid $S$ of $\mathcal{I}_{\mathbb{N}}$ 
generated by $\tilde{a}$ and $\tilde{c}$, as $\tilde{a}\tilde{c}\tilde{a}^{-1}$ is the identity mapping on $\mathbb{N}$. Since $\tilde{a}$ and $\tilde{c}$ 
are not invertible, it follows that $abc$ has no nontrivial invertible pieces, that is, $abc$ is the only invertible piece of itself. 
Furthermore, since $\tilde{c} \neq (\tilde{a}^{-1}\tilde{a})^{-1}$ in $S$, it follows that also $b\neq(ca)^{-1}$ in $M$.

The maximum group image of $M$ is $G=\Gp\pre{a,b,c}{abc=1}$, which is just the free group $FG(a,b)$ with a redundant generator $c=(ab)^{-1}$. The corresponding 
Cayley graph is then, similarly to the example at the very end of Section \ref{sec:univ}, a tree of cycles (``triangles'' in this case, labelled by $abc$), obtained
from the Cayley graph of $FG(a,b)$ with respect to $\{a,b\}$ by completing any path labelled by $ab$ to a 3-cycle by adding a ``backward'' edge labelled $c$.
Therefore, if we want to consider a simple path in this (enriched) Cayley graph from $1$ to an element $g\in G$ ($g\neq 1$), the latter determines a unique sequence of 
3-cycles $C_1,\dots,C_m$ such that $C_1$ contains the vertex $v_0=1$, $C_m$ contains the vertex $v_m=g$, and any two adjacent cycles in the sequence $C_i,C_{i+1}$,
$1\leq i<m$, share just a common vertex $v_i$. This means that there are actually only finitely many simple paths from $1$ to $g$: between each $v_i$ and $v_{i+1}$,
$0\leq i<m$ we have a choice between two subpaths (one of length $1$ and one of length $2$). Hence, the $\sigma$-class of $M(G,\{a,b,c\})$ mapping onto $g$ has
precisely $2^m$ maximal elements. However, we can now single out a simple path $1\leadsto g$ which is in $M$ above every other such path: namely, whenever the
two subpaths between $v_i$ and $v_{i+1}$ are labelled by $b$ and $a^{-1}c^{-1}$, or by $b^{-1}$ and $ca$, always choose the one labelled by $b$ or $b^{-1}$.
For other pairs $v_i,v_{i+1}$ of adjacent vertices the choice is immaterial (which reflects the facts that $a^{-1}=bc$ and $c^{-1}=ab$ hold in $M$), but for the sake
of definiteness always choose the one of length $1$. In this way we get the path $\Gamma_g$ (which is actually the geodesic for $g$), and so if $\phi:M(G,\{a,b,c\})
\to M$ is the canonical homomorphism and $\Pi$ an arbitrary simple path $1\leadsto g$ then $\phi(\Gamma_g,g)\geq\phi(\Pi,g)$. In other words, the $\sigma$-class
of $M$ mapping onto $g$ has a top element, and we conclude that $M$ is $F$-inverse.
\end{exa}

In fact, this example can be generalised, so that by very similar methods we get the following result.

\begin{pro}
For all $k,n\geq 1$, the inverse monoid 
$$M_{k,n}=\Inv\pre{a_1,\dots,a_k}{(a_1\cdots a_k)^n=1}$$ 
is $F$-inverse, but not strongly $F$-inverse.
\end{pro}

Not every one-relator special inverse monoid with a cyclically reduced relator word is $F$-inverse, though, and a counter-example can be found in the recent 
seminal paper \cite{KSz}. Namely, by Proposition 4.8 of that paper, in an $E$-unitary special inverse monoid any two $\sigma$-related elements have a common
upper bound. Therefore, in such an inverse monoid (and this, by \cite[Theorem 4.1]{IMM}, includes all $\Inv\pre{X}{w=1}$ such that $w$ is cyclically reduced), 
this leaves just two possibilities open for a $\sigma$-class: either (i) the $\sigma$-class in question has a unique maximum element, or (ii) it has no maximal 
elements at all, which, by Zorn's Lemma, necessarily means that it contains an ascending chain without an upper bound. Now, the same paper introduces and studies 
the geometric property of an inverse semigroup $S$ called \emph{bounded group distortion} (see \cite[Definition 4.1]{KSz}) that relates the geometry of the 
Cayley graph of the greatest group image of $S$ and that of the \emph{Sch\"utzenberger graphs} \cite{IMM,MM89,Stephen} of $S$ (which are in fact subgraphs of 
the Cayley graph of $S/\sigma$). There are \emph{local} and \emph{uniform} versions of this property, and, as shown in Theorem 4.7 of that paper, in a countably 
generated inverse monoid, the uniform bounded group distortion property is precisely equivalent to each $\sigma$-class having only finitely many maximal elements 
and no ascending chains without upper bounds (i.e.\ the set of maximal elements of each $\sigma$-class is finite and non-empty). Therefore, Theorem 4.9 of the 
same paper states that a countably generated $E$-unitary special inverse monoid is $F$-inverse if and only if it has uniformly bounded group distortion.

\begin{exa}
It follows from \cite[Theorem 4.6]{KSz} and its proof that 
$$
M=\Inv\pre{a,b,c,d}{bcb^{-1}ad^{-1}a^{-1}=1}
$$
is $E$-unitary but \emph{not} $F$-inverse, as it fails to have uniformly bounded group distortion. (Its maximum group image is the free group on $a,b$ and 
$u=bcb^{-1}$.) An alternative way to see that $M$ is not $F$-inverse is to repeatedly use Lemma \ref{lem:ineq} above to deduce the infinite chain of inequalities 
$$
b^{-1}a \leq cb^{-1}ad^{-1} \leq \dots \leq c^kb^{-1}ad^{-k} \leq \dots
$$
between elements of $(b^{-1}a)\sigma$. As N\'ora Szak\'acs \cite{Nora26} pointed out to us, all of the above inequalities are strict; one can compute the Sch\"utzenberger graphs for 
these words (which are certatin closures of the paths from $1$ along these words to their values in $M/\sigma$ in the Cayley graph of $M/\sigma$) and see that they are all different. \end{exa}

This highlights the significance of the following question.

\begin{prb}
Which cyclically reduced words $w\in\ol{X}^*$ have the property that $\Inv\pre{X}{w=1}$ is $F$-inverse?
\end{prb}

However, one-relator special inverse monoids that are (strongly) $F$-inverse are not limited just to those with a cyclically related relator word.

\begin{pro}
Let $M=\Inv\pre{X}{u_i=v_i\; (i\in I)}$ be an inverse monoid with the canonical greatest group image $FG(X)$. Then $M$ is strongly $F$-inverse.
In particular, such is any inverse monoid of the form $\Inv\pre{X}{e_i=f_i\; (i\in I)}$ such that $e_i,f_i\in\ol{X}^*$ are Dyck words, and so
the one-relator special inverse monoid $\Inv\pre{X}{e=1}$ for an arbitrary Dyck word $e\in \ol{X}^*$.
\end{pro}

\begin{proof}
By the given conditions we have $M_{sF}(M/\sigma,X)=M_{sF}(FG(X),X)=FIM(X)$, whence the proposition follows from Corollary \ref{cor:hom}.
\end{proof}

Therefore, we are prompted to pose an even more general question than the one before.

\begin{prb}
Characterise the one-relator special inverse monoids $\Inv\pre{X}{w=1}$ with the (strong) $F$-inverse property.
\end{prb}


\small
\begin{ackn}
The research of the first named author is supported by the Personal Grant F-121 ``Problems of combinatorial semigroup and group theory'' 
of the Serbian Academy of Sciences and Arts. 
The research of the second named author was supported by the Slovenian Research and Innovation Agency grants P1-0288 and J1-60025.
A big part of this research was conducted during the first author's two visits to the Institute of Mathematics, Physics and Mechanics (IMFM), 
Ljubljana, Slovenia, in November 2024 and February 2026, which was partially supported by the Slovenian Research and Innovation Agency grant P1-0288 and J1-60025.
\end{ackn}
\normalsize



\begin{thebibliography}{99}
\frenchspacing

\bibitem{ABO}
K. Auinger, J. Bitterlich, M. Otto, 
Finite approximation of free groups I: The $F$-inverse cover problem,
\emph{Adv. Math.} \textbf{482} (2025), Article No.110563, 54 pp.

\bibitem{AKSz}
K. Auinger, G. Kudryavtseva, M. B. Szendrei, 
$F$-inverse monoids as algebraic structures in enriched signature,
\emph{Indiana Univ. Math. J.} \textbf{70} (2021), 2107--2131.

\bibitem{BJM}
F. Berthaut, D. Janin, B. Martin,
Advanced synchronization of audio or symbolic musical patterns,
in: \emph{6th IEEE International Conference on Semantic Computing}, pp.\ 202-–209, 
IEEE Society Press, 2012.

\bibitem{CP}
A. H. Clifford, G. B. Preston,
\emph{The Algebraic Theory of Semigroups, Vol. I \& II},
Mathematical Surveys No. 7, American Mathematical Society, Providence, R.I., 1961 \& 1967.

\bibitem{DG21}
I. Dolinka, R. D. Gray, 
New results on the prefix membership problem for one-relator groups, 
\emph{Trans. Amer. Math. Soc.} \textbf{374} (2021), 4309--4358.

\bibitem{Gr-Inv}
R. D. Gray,
Undecidability of the word problem for one-relator inverse monoids via right-angled 
Artin subgroups of one-relator groups,
\emph{Invent. Math.} \textbf{219} (2020), 987--1008.

\bibitem{GR}
R. D. Gray, N. Ru\v skuc, 
On groups of units of special and one-relator inverse monoids,
\emph{J. Inst. Math. Jussieu} \textbf{23} (2024), 1875--1918.

\bibitem{HR}
K. Henckell, J. Rhodes, 
The theorem of Knast, the $PG = BG$ and type-II conjectures,
in: \emph{Monoids and semigroups with applications} (Berkeley, CA, 1989), pp.\ 453--463,
World Scientific, River Edge, NJ, 1991.

\bibitem{How}
J. M. Howie,
\emph{Fundamentals of Semigroup Theory}, 
London Math. Soc. Monographs Vol. 12, Clarendon Press, Oxford, 1995.

\bibitem{IMM}
S. V. Ivanov, S. W. Margolis, J. C. Meakin,
On one-relator inverse monoids and one-relator groups,
\emph{J. Pure Appl. Algebra} \textbf{159} (2001), 83--111.

\bibitem{Janin}
D. Janin,
Vers une mod\'elisation combinatoire des structures rythmiques simples de la musique,
\emph{Revue Francophone d'Informatique et Musique} \textbf{2} (2012), 40 pp.

\bibitem{KSz}
M. Kambites, N. Szak\'acs,
The large scale geometry of inverse semigroups and their maximal group images,
\emph{Geom. Dedicata} \textbf{220} (2026), Article No.19, 30 pp.

\bibitem{Kinyon}
M. Kinyon, 
$F$-inverse semigroups as $\langle 2,1,1\rangle$-algebras,
talk at the \emph{International Conference on Semigroups}, Lisbon, 2018.

\bibitem{KLF}
G. Kudryavtseva, A. Lemut Furlani,
A new approach to universal $F$-inverse monoids in enriched signature,
\emph{Results Math.} \textbf{79} (2024), Article No.260, 13 pp.

\bibitem{Law}
M. V. Lawson,
\emph{Inverse Semigroups: The Theory of Partial Symmetries},
World Scientific, Singapore, 1998.

\bibitem{LSch}
R. C. Lyndon, P. E. Schupp,
\emph{Combinatorial Group Theory},
Springer-Verlag, Berlin, 1977. 

\bibitem{Ma1}
W. Magnus, 
\"Uber diskontinuierliche Gruppen mit einer definierenden Relation. (Der Freiheitssatz),
\emph{J. reine angew. Math.} \textbf{163} (1930), 141--165.

\bibitem{Ma2}
W. Magnus, 
Das Identit\"atsproblem f\"ur Gruppen mit einer definierenden Relation, 
\emph{Math. Ann.} \textbf{106} (1932), 295--307.

\bibitem{MM89}
S. W. Margolis, J. C. Meakin,
$E$-unitary inverse monoids and the Cayley graph of a group presentation,
\emph{J. Pure Appl. Algebra} \textbf{58} (1989), 45--76.

\bibitem{Mea}
J. Meakin, 
Groups and semigroups: connections and contrasts,
in: \emph{Groups St Andrews 2005, Vol. 2}, pp. 357--400,
London Math. Soc. Lecture Note Ser., Vol. 340, Cambridge Univ. Press, Cambridge, 2007.

\bibitem{Munn}
W. D. Munn,
Free inverse semigroups, 
\emph{Proc. London Math. Soc. (3)} \textbf{29} (1974), 385--404.

\bibitem{CF}
C.-F. Nyberg Brodda,
The word problem for one-relation monoids: a survey,
\emph{Semigroup Forum} \textbf{103} (2021), 297--355.

\bibitem{Pet}
M. Petrich,
\emph{Inverse Semigroups},
Wiley, 1984.

\bibitem{Sch}
H. E. Scheiblich,
Free inverse semigroups, 
\emph{Proc. Amer. Math. Soc.} \textbf{38} (1973), 1--7.

\bibitem{Stephen}
J. B. Stephen,
Presentations of inverse semigroups,
\emph{J. Pure Appl. Algebra} \textbf{63} (1990), 81--112.

\bibitem{Nora}
N. Szak\'acs,
$E$-unitary and $F$-inverse monoids, and closure operators on group Cayley graphs,
\emph{Acta Math. Hungar.} \textbf{173} (2024), 297--316.

\bibitem{Nora26}
N. Szak\'acs, Personal communication, 2026.

\bibitem{W1}
C. M. Weinbaum, 
Visualizing the word problem with an application to sixth groups,
\emph{Pacific J. Math.} \textbf{16} (1966), 557--578.

\bibitem{W2}
C. M. Weinbaum, 
On relators and diagrams for groups with one defining relation,
\emph{Illinois J. Math.} \textbf{16} (1972), 308--322.

\end{thebibliography}
\end{document}